\documentclass[pagesize,pdftex]{scrartcl}
\usepackage{latexsym} 
\usepackage[intlimits]{amsmath}
\usepackage{amsthm}
\usepackage{amsfonts}
\usepackage{amssymb}
\usepackage{amsxtra}
\usepackage{amscd}
\usepackage{ifthen}
\usepackage{graphicx}
\usepackage{color}
\usepackage[shortlabels]{enumitem}
\usepackage{mathrsfs}
\usepackage[pagebackref=true]{hyperref}
\usepackage[arrow, matrix, curve]{xy}
\usepackage{doi}

\hypersetup{
pdfauthor={Aday Celik, Mads Kyed},
pdftitle={Nonlinear acoustics: On the Blackstock-Crighton equations with periodic forcing term},
breaklinks=true,
colorlinks=true,
linkcolor=blue,
citecolor=blue,
urlcolor=blue,
filecolor=blue,
}

\numberwithin{equation}{section} 
\setkomafont{title}{\normalfont}

\newenvironment{pdeq}{ \left\{ \begin{aligned}}{\end{aligned}\right.}

\newcommand{\np}[1]{(#1)}
\newcommand{\nb}[1]{[#1]}
\newcommand{\bp}[1]{\big(#1\big)}
\newcommand{\bb}[1]{\big[#1\big]}
\newcommand{\Bp}[1]{\bigg(#1\bigg)}

\newcommand{\Bb}[1]{\bigg[#1\bigg]}
\newcommand{\calp}{{\mathcal P}}

\newcommand{\calt}{{\mathcal T}}

\newcommand{\R}{\mathbb{R}}
\newcommand{\C}{\mathbb{C}}
\newcommand{\Z}{\mathbb{Z}}

\newcommand{\N}{\mathbb{N}}
\newcommand{\embeds}{\hookrightarrow}

\newcommand{\TD}{\operatorname{Tr}_{D}}

\newcommand{\TDN}{\operatorname{Tr}_{0}}
\newcommand{\TDNrinv}{\operatorname{R}_{0}}
\newcommand{\TN}{\operatorname{Tr}_N}

\newcommand{\ADdiffopru}{\left(a\Delta - \partial_t\right)\left(\partial_{t}^2\uvel-c^2\Delta\uvel - b\partial_t\Delta\uvel\right)}

\newcommand{\ra}{\rightarrow}

\newcommand{\set}[1]{\ensuremath{\{#1\}}}

\newcommand{\setc}[2]{\ensuremath{\{#1\ \lvert\ #2\}}}

\newcommand{\setcL}[2]{\ensuremath{\biggl\{#1\ \lvert\ #2\biggr\}}}

\newcommand{\closure}[2]{\overline{#1}^{#2}}
\newcommand{\proj}{\calp}
\newcommand{\projcompl}{\calp_\bot}

\newcommand{\quotientmap}{\pi}
\newcommand{\torus}{{\mathbb T}}

\newcommand{\grad}{\nabla}

\newcommand{\dx}{{\mathrm d}x}

\newcommand{\ds}{{\mathrm d}s}
\newcommand{\dt}{{\mathrm d}t}

\newcommand{\dS}{{\mathrm d}S}

\newcommand{\TDR}{\mathscr{S^\prime}}

\newcommand{\FT}{\mathscr{F}}
\newcommand{\iFT}{\mathscr{F}^{-1}}

\newcommand{\PR}{\mathcal{P}}
\newcommand{\oPR}{\mathcal{P}_\bot}
\newcommand{\mmultiplier}{m}

\newcommand{\norm}[1]{\lVert#1\rVert}

\newcommand{\normL}[1]{\Bigl\lVert#1\Bigr\rVert}
\newcommand{\snorm}[1]{{\lvert #1 \rvert}}

\newcommand{\opnorm}[1]{{\vert\kern-0.25ex\vert\kern-0.25ex\vert #1 \vert\kern-0.25ex\vert\kern-0.25ex\vert}}
\newcommand{\WSR}[2]{W^{#1,#2}}

\newcommand{\CR}[1]{C^{#1}}  
\newcommand{\LR}[1]{L^{#1}}

\newcommand{\CRi}{\CR \infty}

\newcommand{\PS}[1]{\mathbb{X}_{\mathrm{per}}^{#1}}
\newcommand{\PSnp}[1]{\mathbb{X}^{#1}}

\newcommand{\TRD}[1]{\operatorname{T}^p_{\mathrm{per},D_{#1}}}
\newcommand{\TRN}[1]{\operatorname{T}^p_{\mathrm{per},N_{#1}}}

\newcommand{\LRper}[1]{L^{#1}_{\mathrm{per}}}
\newcommand{\WSRper}[2]{W^{#1,#2}_{\mathrm{per}}} 
\newcommand{\CRper}{\CR{}_{\mathrm{per}}}
\newcommand{\CRiper}{\CR{\infty}_{\mathrm{per}}}

\newcommand{\wvel}{w}

\newcommand{\uvel}{u}
\newcommand{\us}{\uvel_s}
\newcommand{\up}{\uvel_p}

\newcommand{\tin}{\text{in }}
\newcommand{\tif}{\text{if }}
\newcommand{\ton}{\text{on }}

\renewcommand{\epsilon}{\varepsilon}

\newcommand{\tay}{\calt}
\newcommand{\per}{\tay}

\newcommand{\newCCtr}[2][d]{
\newcounter{#2}\setcounter{#2}{0}
\expandafter\xdef\csname kyedtheconst#2\endcsname{#1}
}
\newcommand{\Cc}[2][nolabel]{
\stepcounter{#2}
\expandafter\ensuremath{\csname kyedtheconst#2\endcsname_{\arabic{#2}}}
\ifthenelse{\equal{#1}{nolabel}}
{}
{\expandafter\xdef\csname kyedconst#1\endcsname
{\expandafter\ensuremath{\csname kyedtheconst#2\endcsname_{\arabic{#2}}}}}
}
\newcommand{\Ccn}[2][nolabel]{
\expandafter\ensuremath{\csname kyedtheconst#2\endcsname}
\ifthenelse{\equal{#1}{nolabel}}
{}
{\expandafter\xdef\csname kyedconst#1\endcsname
{\expandafter\ensuremath{\csname kyedtheconst#2\endcsname}}}
}
\newcommand{\CcSetCtr}[2]{
\setcounter{#1}{#2}
}
\newcommand{\Cclast}[1]{
\expandafter\ensuremath{\csname kyedtheconst#1\endcsname_{\arabic{#1}}}
}
\newcommand{\Ccllast}[1]{
\addtocounter{#1}{-1}
\expandafter\ensuremath{\csname kyedtheconst#1\endcsname_{\arabic{#1}}}
\addtocounter{#1}{1}
}
\newcommand{\const}[1]{
\expandafter{\ifcsname kyedconst#1\endcsname
  \csname kyedconst#1\endcsname
\else
  \errmessage{Undefined Kyedconstant #1.}%
\fi}
}

\theoremstyle{plain}
\newtheorem{thm}{Theorem}[section]

\newtheorem{lem}[thm]{Lemma}

\theoremstyle{remark}

\begin{document}
\title{Nonlinear acoustics: Blackstock-Crighton equations with a periodic forcing term}

\author{
Aday Celik\\ 
Fachbereich Mathematik\\
Technische Universit\"at Darmstadt\\
Schlossgartenstr. 7, 64289 Darmstadt, Germany\\
Email: {\texttt{celik@mathematik.tu-darmstadt.de}}
\and
Mads Kyed\\ 
Fachbereich Mathematik\\
Technische Universit\"at Darmstadt\\
Schlossgartenstr. 7, 64289 Darmstadt, Germany\\
Email: {\texttt{kyed@mathematik.tu-darmstadt.de}}
}

\date{\today}
\maketitle

\begin{abstract}
The Blackstock-Crighton equations describe the motion of a viscous, heat-conducting, compressible fluid. They are used as models for acoustic wave propagation in a medium in which both nonlinear and dissipative effects are taken into account. In this article, a mathematical analysis of the Blackstock-Crighton equations with a time-periodic forcing term is carried out. For arbitrary time-periodic data (sufficiently restricted in size) it is shown that a time-periodic solution of the same period always exists. This implies that the dissipative effects are sufficient to avoid resonance within the Blackstock-Crighton models.
The equations are considered in a three-dimensional bounded domain with both non-homogeneous Dirichlet and Neumann boundary values. Existence of a solution is obtained via a fixed-point argument based on appropriate a priori estimates for the linearized equations.
\end{abstract}


\newCCtr[C]{C}
\newCCtr[M]{M}
\newCCtr[B]{B}
\newCCtr[\epsilon]{eps}
\CcSetCtr{eps}{-1}
\newCCtr[c]{c}
\let\oldproof\proof
\def\proof{\CcSetCtr{c}{-1}\oldproof}

\section{Introduction}

The motion of a viscous, heat-conducting fluid is governed by mass conservation, momentum conservation, energy conservation, and a thermodynamic equation of state. The compressible Navier-Stokes equations describe the conservation of mass and momentum when viscous effects are taken into account.
The Kirchhoff-Fourier equations describe the conservation of energy when 
heat-conducting effects are taken into account. Using the equation of state of an ideal fluid, and assuming the flow is irrotational, \textsc{Blackstock} \cite{Bla63} 
eliminated all but one dependent variable from the Navier-Stokes and Kirchhoff-Fourier equations to obtain  
the following model for a viscous, heat-conducting fluid:
\begin{align}\label{BCKE}
\left(a\Delta - \partial_t\right)\left(\partial_{t}^2\uvel-c^2\Delta\uvel - b\partial_t\Delta\uvel\right) - 
\partial_t^2\bp{\frac{1}{c^2}\frac{B}{2A}\np{\partial_t\uvel}^2 + \snorm{\grad\uvel}^2} &= f.\tag{BCK}
\end{align}
Here, $\uvel$ denotes the potential of the fluids velocity field and $f$ a forcing term.  
The constant $a$ is the heat conductivity of the fluid, $b$ the diffusivity of sound, $c$ the speed of sound, and $B/A$ is called the parameter of nonlinearity.
The equation \eqref{BCKE} is referred to as the Blackstock-Crighton-Kuznetsov equation and is used as a model for acoustic wave propagation in a medium in which both nonlinear and dissipative effects are taken into account. 
If local nonlinear effects are neglected, one is lead to the  
Blackstock-Crighton-Westervelt equation
\begin{align}\label{BCWE}
\left(a\Delta - \partial_t\right)\left(\partial_{t}^2\uvel-c^2\Delta\uvel - b\partial_t\Delta\uvel\right) - \partial_t^2\Bp{\frac{1}{c^2}\bp{1 + \frac{B}{2A}}\np{\partial_t\uvel}^2} &= f.\tag{BCW}
\end{align}
If also the dissipative effects are neglected, then the model reduces to the classical wave equation.

In the present article we investigate if the dissipative effects present in \eqref{BCKE} and \eqref{BCWE} are sufficient to avoid resonance effects. If one considers the classical (hyperbolic) wave equation, where dissipative effects are completely neglected, 
a time-periodic source term would inevitable lead to resonance, that is, an unbounded solution. 
Our investigation of \eqref{BCKE} and \eqref{BCWE} will show, by establishing existence of a time-periodic solution
for arbitrary time-periodic data $f$ (sufficiently restricted in size), that this is not the case for 
the Blackstock-Crighton equations \eqref{BCKE} and \eqref{BCWE}.
We may therefore conclude that the dissipative effects of viscosity and heat conduction 
in the Blackstock-Crighton framework constitute a sufficient energy absorption mechanism to avoid resonance.

The initial-value problems corresponding to the Blackstock-Crighton equations \eqref{BCKE} and \eqref{BCWE} have 
been investigated recently by \textsc{Brunnhuber} and \textsc{Kaltenbacher} \cite{BK14}, \textsc{Brunnhuber} \cite{Bru15} 
(in particular we refer to \cite[Section 2]{Bru15} for a derivation of \eqref{BCKE}), and
\textsc{Brunnhuber} and \textsc{Meyer} \cite{BM16}. To our knowledge, the investigation of time-periodic solutions to the Blackstock-Crighton equations 
in a setting of time-periodic data is new. The Kuznetsov equation, which is a simpler model for nonlinear acoustic wave propagation, was investigated in a time-periodic setting by the present authors in \cite{CelikKyed_nweqwd}. We shall employ the results from \cite{CelikKyed_nweqwd} in the following. 
In addition to the articles mentioned above, we would like to draw the readers attention to the recent work \cite{Tani} of \textsc{Tani}, in which a new model for nonlinear wave propagation similar to the Blackstock-Crighton model is derived and analyzed. 

We consider \eqref{BCKE} and \eqref{BCWE} with inhomogeneous Dirichlet and Neumann boundary conditions.
We shall work in a setting of time-periodic functions and therefore take the whole of $\R$ as time-axis. We let $\Omega\subset\R^3$ denote a spatial domain. In the following, $(t,x)\in\R\times\Omega$ will always denote a time-variable $t$ and spatial variable $x$, respectively.   
The Blackstock-Crighton equation (in generalized form) with Dirichlet boundary condition then reads
\begin{align}\label{BCD}
\begin{pdeq}
\left(a\Delta - \partial_t\right)\left(\partial_{t}^2\uvel-c^2\Delta\uvel - b\partial_t\Delta\uvel\right) - \partial_t^2\bp{k\np{\partial_t\uvel}^2 + s\snorm{\grad\uvel}^2}&= f && \tin\R\times\Omega, \\
\bp{\uvel, \Delta\uvel} &= \bp{g, h} && \ton\R\times\partial\Omega,
\end{pdeq}\tag{BCD}
\end{align}
where we have used the same notation $k:=\frac{1}{c^2}\big((1-s)+\frac{B}{2A}\big)$ and $s\in\{0, 1\}$ as in \cite{Bru15}. 
The corresponding Neumann problem reads
\begin{align}\label{BCN}
\begin{pdeq}
\left(a\Delta - \partial_t\right)\left(\partial_{t}^2\uvel-c^2\Delta\uvel - b\partial_t\Delta\uvel\right) - \partial_t^2\bp{k\np{\partial_t\uvel}^2 + s\snorm{\grad\uvel}^2}&= f && \tin\R\times\Omega, \\
\bp{\partial_\nu\uvel, \partial_\nu{\Delta\uvel}} &= \bp{g, h} && \ton\R\times\partial\Omega.
\end{pdeq}\tag{BCN}
\end{align}
Here $f\colon\R\times\Omega\to\R$ and $g, h\colon\R\times\partial\Omega\to\R$ are given, and $\uvel\colon\R\times\Omega\to\R$ is the unknown. 

We shall consider data that are time-periodic with the period $\per>0$, that is, functions satisfying
\begin{align*}
\forall t\in\R:\quad v(t+\per,\cdot)=v(t,\cdot).
\end{align*}
As the main result in this article, we show for given $\per$-time-periodic data $f$, $g$ and $h$ in appropriate function spaces and sufficiently restricted in size 
that there exists a time-periodic solution $\uvel$ to \eqref{BCD} and \eqref{BCN}. 
In the case of Dirichlet boundary conditions, the result can be stated as follows (the subscript $per$ indicates that a 
function space consists of $\per$-time-periodic functions; see Section \ref{FunctionSpacesSection}):
\begin{thm}\label{max_regBCD}
Let $\Omega\subset\R^3$ be a bounded domain with a $C^4$-smooth boundary and $p\in (\frac{5}{2}, \infty)$. There is an $\varepsilon>0$ such that for all 
$f\in\LRper{p}(\R; \LR{p}\left(\Omega\right))$, $g\in\TRD{1}(\R\times\partial\Omega)$ and $h\in\TRD{2}(\R\times\partial\Omega)$ satisfying 
\begin{align*}
\norm{f}_{\LRper{p}(\R; \LR{p}(\Omega))} + \norm{g}_{\TRD{1}(\R\times\partial\Omega)} + \norm{h}_{\TRD{2}(\R\times\partial\Omega)} \leq \varepsilon
\end{align*}
there is a solution
\begin{align*}
\uvel\in\WSRper{3}{p}(\R; \LR{p}(\Omega)) \cap \WSRper{1}{p}\left(\R; \WSR{4}{p}(\Omega)\right)
\end{align*}
to \eqref{BCD}.
\end{thm}

A solution to the Neumann problem \eqref{BCN} only exists when the data satisfies certain compatibility conditions. More precisely, we obtain:
\begin{thm}\label{max_regBCN}
Let $\Omega$ and $p$ be as in Theorem \ref{max_regBCD}. There is an $\varepsilon>0$ such that for all $f\in\LRper{p}(\R; \LR{p}\left(\Omega\right))$, $g\in \TRN{1}(\R\times\partial\Omega)$ and $h\in \TRN{2}(\R\times\partial\Omega)$ satisfying 
\begin{align*}
\norm{f}_{\LRper{p}(\R; \LR{p}(\Omega))} + \norm{g}_{\TRN{1}(\R\times\partial\Omega)} + \norm{h}_{\TRN{2}(\R\times\partial\Omega)} \leq \varepsilon
\end{align*}
and
\begin{align*}
\int_0^\per\int_\Omega f \,\dx\dt + ac^2\int_0^\per\int_{\partial\Omega} h \,\dS\dt = 0 
\end{align*}
there is a solution 
\begin{align*}
\uvel\in \WSRper{3}{p}(\R; \LR{p}(\Omega)) \cap \WSRper{1}{p}\left(\R; \WSR{4}{p}(\Omega)\right)
\end{align*}
to \eqref{BCN}.
\end{thm}

Our proofs are based on a priori $\LR{p}$-estimates for the corresponding linearizations 
\begin{align}\label{LinBCD}
\begin{pdeq}
\left(a\Delta - \partial_t\right)\left(\partial_{t}^2\uvel-c^2\Delta\uvel - b\partial_t\Delta\uvel\right) &= f && \tin\R\times\Omega, \\
\bp{\uvel, \Delta\uvel} &= \bp{g, h} && \ton\R\times\partial\Omega,
\end{pdeq}
\end{align}
and
\begin{align}\label{LinBCN}
\begin{pdeq}
\left(a\Delta - \partial_t\right)\left(\partial_{t}^2\uvel-c^2\Delta\uvel - b\partial_t\Delta\uvel\right) &= f && \tin\R\times\Omega, \\
\bp{\partial_\nu\uvel, \partial_\nu{\Delta\uvel}} &= \bp{g, h} && \ton\R\times\partial\Omega,
\end{pdeq}
\end{align}
of \eqref{BCD} and \eqref{BCN}, respectively, and an application of the contraction mapping principle. 
Our analysis relies on a decomposition of \eqref{LinBCD} and \eqref{LinBCN} into a coupled system consisting of the time-periodic heat equation and 
the time-periodic Kuznetsov equation. The time-periodic heat equation was investigate in \cite{KyedSauer_Heat}, and 
the time-periodic Kuznetsov equation in \cite{CelikKyed_nweqwd}. In the following, we shall employ the results obtained in these two articles.

\section{Function Spaces}\label{FunctionSpacesSection}

In the following, $\Omega\subset\R^3$ shall always denote a three-dimensional domain with a $\CR{4}$-smooth boundary.
The outer normal on $\partial\Omega$ is denoted by $\nu$.
Points in $\R\times\Omega$ are generally denoted by $(t,x)$, with $t$ being referred to as time, and $x$ as the spatial variable. 
A time period $\per>0$ remains fixed.

Classical Lebesgue and Sobolev spaces are denoted by $\LR{p}(\Omega)$ and $\WSR{k}{p}(\Omega)$, respectively. We write $\norm{\cdot}_p$ and $\norm{\cdot}_{k,p}$ instead of $\norm{\cdot}_{\LR{p}(\Omega)}$ and $\norm{\cdot}_{\WSR{k}{p}(\Omega)}$. 

For a Lebesgue or Sobolev space $E(\Omega)$, we define the space of smooth $\per$-time-periodic vector-valued functions by 
\begin{align*}
\CRiper\bp{\R; E(\Omega)}:= \setc{f\in\CRi\bp{\R; E(\Omega)}}{f(t+\per, x) = f(t, x) }.
\end{align*}
For $p\in(1, \infty)$, we let
\begin{align}
&\norm{f}_{\LRper{p}\left(\R; E(\Omega)\right)}:=\Bp{\frac{1}{\per}\int_0^\per \norm{f(t)}^p_{E(\Omega)} \,\dt}^{\frac{1}{p}},\label{LppernormDef} \\
&\norm{f}_{\WSRper{k}{p}\left(\R; E(\Omega)\right)}:=\Bp{\sum_{\alpha = 0}^k \norm{\partial_t^\alpha f}_{\LRper{p}\left(\R; E(\Omega)\right)}^p}^{\frac{1}{p}}.\label{WSRpernormDef}
\end{align}
As one may verify, $\norm{\cdot}_{\LRper{p}\left(\R; E(\Omega)\right)}$ and $\norm{\cdot}_{\WSRper{k}{p}\left(\R; E(\Omega)\right)}$ define norms on $\CRiper\bp{\R; E(\Omega)}$. 
We put
\begin{align}
&\LRper{p}\bp{\R; E(\Omega)}:= \closure{\CRiper\bp{\R; E(\Omega)}}{\norm{\cdot}_{\LRper{p}\left(\R; E(\Omega)\right)}}.\label{DefLpperSpace}
\end{align}
If no confusion can arise, we write $\norm{\cdot}_p$ instead of $\norm{\cdot}_{\LRper{p}(\R; \LR{p}(\Omega))}$.
We also introduce Sobolev spaces of vector-valued time-periodic functions:
\begin{align}
&\WSRper{k}{p}\bp{\R; E(\Omega)} := \closure{\CRiper\bp{\R; E({\Omega}{})}}{\norm{\cdot}_{\WSRper{k}{p}\left(\R; E(\Omega)\right)}}.\label{DefWSRperSpace}
\end{align}
Corresponding Sobolev-Slobodecki\u{\i} spaces are defined in the usual way by real interpolation $(k\in\N_0,\alpha\in(0,1))$:
\begin{align*}
&\WSRper{k+\alpha}{p}\bp{\R; E(\Omega)} := \bp{\WSRper{k+1}{p}\bp{\R; E(\Omega)},\WSRper{k}{p}\bp{\R; E(\Omega)}  }_{1-\alpha,p}.
\end{align*}
Moreover, we let
\begin{align*}
&\PS{p}(\R\times\Omega) := \WSRper{3}{p}\bp{\R; \LR{p}(\Omega)} \cap \WSRper{1}{p}\bp{\R; \WSR{4}{p}(\Omega)}.
\end{align*}
Embedding properties are collected in the following lemma.

\begin{lem}\label{EmbeddingLemma}
Let $\Omega\subset\R^3$ be a bounded domain with a $C^{4}$-smooth boundary and $p\in(1,\infty)$. The embeddings 
\begin{align}
&\PS{p}(\R\times\Omega) \embeds 
\WSRper{2}{p}\bp{\R;\WSR{2}{p}(\Omega)},\label{EmbeddingLemEmb1}\\
&\WSRper{2}{p}\bp{\R;\LR{p}(\Omega)}\cap \LRper{p}\bp{\R;\WSR{4}{p}(\Omega)} \embeds
\WSRper{1}{p}\bp{\R;\WSR{2}{p}(\Omega)}\label{EmbeddingLemEmb2}
\end{align}
and $(l\in\{1,2,3\})$
\begin{align}
&\begin{aligned}
&\LR{p}\bp{\R_+; \PS p(\R\times\R^2)}\cap\WSR{4}{p}\bp{\R_+; \WSRper{1}{p}\bp{\R; \LR{p}(\R^2)}} \\
&\qquad\qquad\qquad\qquad\qquad\qquad\qquad\qquad\qquad\embeds
\WSR{l}{p}\bp{\R_+; \WSRper{1}{p}\bp{\R; \WSR{4-l}{p}(\R^2)}}
\end{aligned}\label{EmbeddingLemEmb5}
\end{align}
are continuous.
\end{lem}

\begin{proof}
The regularity of $\Omega$ suffices to ensure the existence of a continuous extension operator $E\colon\PS{p}(\R\times\Omega)\to\PS{p}(\R\times\R^3)$ as in the case of classical Sobolev spaces. Consequently, it suffices to prove Lemma \ref{EmbeddingLemma} 
in the whole-space case $\Omega=\R^3$. For this purpose, it is convenient to replace the time-axis $\R$ with the torus 
$\torus:=\R/\per\Z$ in the function spaces of $\per$-time-periodic functions. The torus $\torus$ canonically inherits a topology and
differentiable structure from $\R$ via the 
quotient mapping $\quotientmap:\R\ra\torus$ in such a way that 
\begin{align*}
\CRi\bp{\torus; E(\R^3)} = \setc{f:\torus\ra E(\R^3)}{f\circ\quotientmap\in\CRi(\R;E(\R^3))}
\end{align*}
for any generic Sobolev space $E(\R^3)$.
Moreover, if $\torus$ is equipped with the normalized Haar measure, we may introduce the norms 
$\norm{\cdot}_{\LR{p}(\torus;E(\R^3))}$ and 
$\norm{\cdot}_{\WSR{k}{p}(\torus;E(\R^3))}$ on $\CRi\bp{\torus; E(\R^3)}$ by the same expressions as in \eqref{LppernormDef}--\eqref{WSRpernormDef}. Lebesgue spaces $\LR{p}\bp{\torus;E(\R^3)}$ and Sobolev spaces 
$\WSR{k}{p}\bp{\torus;E(\R^3)}$ on the torus are then defined as in \eqref{DefLpperSpace} and \eqref{DefWSRperSpace}, respectively. 
The quotient map $\quotientmap$ employed as a lifting operator acts as an isometric isomorphism between $\CRiper\bp{\R; E(\R^3)}$ and $\CRi\bp{\torus; E(\R^3)}$, and consequently also between the Sobolev spaces 
$\WSRper{k}{p}\bp{\R;E(\R^3)}$ and 
$\WSR{k}{p}\bp{\torus;E(\R^3)}$. 
We may therefore verify the embeddings in the setting of function spaces where the time-axis has been replaced with the torus. In this setting, we can employ the Fourier transform $\FT_{\torus\times\R^3}$ in time \emph{and} space to characterize the Sobolev spaces 
\begin{multline*}
\WSR{2}{p}\bp{\torus;\WSR{2}{p}(\R^3)} \\= \setcL{f\in\TDR\np{\torus\times\R^3}}{\iFT_{\torus\times\R^3}\bb{\bp{1+\snorm{k}^2\snorm{\xi}^2}\FT_{\torus\times\R^3}\nb{f}}\in\LR{p}(\torus\times\R^3)}
\end{multline*}
and
\begin{align*}
\PSnp{p}(\torus\times\R^3) =& \WSR{3}{p}\bp{\torus; \LR{p}(\R^3)} \cap \WSR{1}{p}\bp{\torus; \WSR{4}{p}(\R^3}\\
=& \setcL{f\in\TDR\np{\torus\times\R^3}}{\iFT_{\torus\times\R^3}\bb{\bp{1+\snorm{k}^3+\snorm{k}\snorm{\xi}^4}\FT_{\torus\times\R^3}\nb{f}}\in\LR{p}(\torus\times\R^3)}.
\end{align*}
Here, $\TDR\np{\torus\times\R^3}$ denotes the space of Schwartz-Bruhat distributions on the locally compact abelian group $\torus\times\R^3$, and $(k,\xi)\in\Z\times\R^3$ an element of the corresponding dual group.
Now observe that
\begin{multline}\label{MultiplierIdEmb1}
\norm{f}_{\WSR{2}{p}\np{\torus;\WSR{2}{p}(\R^3)}} \\\leq C\, 
\normL{\iFT_{\torus\times\R^3}\Bb{\frac{1+\snorm{k}^2\snorm{\xi}^2}{1+\snorm{k}^3+\snorm{k}\snorm{\xi}^4}
\FT_{\torus\times\R^3}\bb{\iFT_{\torus\times\R^3} \nb{(1+\snorm{k}^3+\snorm{k}\snorm{\xi}^4)\FT_{\torus\times\R^3}\nb{f}}}}}_p.
\end{multline}
The multiplier
\begin{align*}
\mmultiplier:\R\times\R^3\ra\C,\quad \mmultiplier(\eta,\xi):= \frac{1+\snorm{\eta}^2\snorm{\xi}^2}{1+\snorm{\eta}^3+\snorm{\eta}\snorm{\xi}^4}
\end{align*}
satisfies the condition of the Marcinkiewicz's multiplier theorem (\cite[Chapter IV, \S 6]{Stein70}). Indeed, 
by Young's inequality $\snorm{\eta}^2\snorm{\xi}^2\leq \frac{1}{3}\snorm{\eta}^3+\frac{2}{3}\snorm{\eta}\snorm{\xi}^4$, whence 
$\norm{\mmultiplier}_\infty<\infty$. Similarly, one may verify that  
\begin{align*}
\max_{\epsilon\in\set{0,1}^{n+1}} \norm{\xi_1^{\epsilon_1}\cdots\xi_n^{\epsilon_n}\eta^{\epsilon_{n+1}}
\partial_{1}^{\epsilon_1}\cdots\partial_{n}^{\epsilon_n}\partial_\eta^{\epsilon_{n+1}}
\mmultiplier(\eta,\xi)}_\infty < \infty.
\end{align*}
Consequently, $\mmultiplier$ is an $\LR{p}(\R\times\R^3)$ multiplier. 
By de Leeuw's Transference Principle for Fourier multipliers an locally compact abelian groups (see for example \cite[Theorem B.2.1]{EdG77}), it follows that the restriction $\mmultiplier_{|\Z\times\R^3}$ is an $\LR{p}(\torus\times\R^3)$ multiplier. From \eqref{MultiplierIdEmb1} we thus deduce 
\begin{align*}
\norm{f}_{\WSR{2}{p}\np{\torus;\WSR{2}{p}(\R^3)}} \leq C\,
\normL{\iFT_{\torus\times\R^3} \nb{(1+\snorm{k}^3+\snorm{k}\snorm{\xi}^4)\FT_{\torus\times\R^3}\nb{f}}}_{\LR{p}(\torus\times\R^3)}\leq 
C\,\norm{f}_{\PSnp{p}(\torus\times\R^3)},
\end{align*}
and we conclude \eqref{EmbeddingLemEmb1}.
The embeddings \eqref{EmbeddingLemEmb2}--\eqref{EmbeddingLemEmb5} can be established in a completely similar manner. 
\end{proof}

Additional, we make use of the following embedding properties, which have already been established in \cite{GaldiKyed}:

\begin{lem}\label{SobEmbeddingThm}
Let $\Omega\subset\R^3$ be a bounded domain with a $C^{2}$-smooth boundary and $p\in(1,\infty)$. 
Assume that $\alpha\in\bb{0,2}$ and $p_0,r_0\in[q,\infty]$ satisfy
\begin{align*}
\begin{pdeq}
&r_0\leq \frac{2q}{2-\alpha q} && \tif\ \alpha q<2,\\
&r_0<\infty && \tif\ \alpha q =2,\\
&r_0\leq\infty && \tif\ \alpha q >2,
\end{pdeq}
\qquad
\begin{pdeq}
&p_0\leq \frac{nq}{n-(2-\alpha) q} && \tif\  \np{2-\alpha}q<{n},\\
&p_0<\infty && \tif\ \np{2-\alpha}q={n},\\
&p_0\leq\infty && \tif\ \np{2-\alpha}q>{n},
\end{pdeq}
\end{align*}
and that $\beta\in\bb{0,1}$ and $p_1,r_1\in[q,\infty]$ satisfy
\begin{align*}
\begin{pdeq}
&r_1\leq \frac{2q}{2-\beta q} && \tif\ \beta q<2,\\
&r_1<\infty && \tif\ \beta q =2,\\
&r_1\leq\infty && \tif\ \beta q >2,
\end{pdeq}
\qquad
\begin{pdeq}
&p_1\leq \frac{nq}{n-(1-\beta) q} && \tif\  \np{1-\beta}q<{n},\\
&p_1<\infty && \tif\ \np{1-\beta}q={n},\\
&p_1\leq\infty && \tif\ \np{1-\beta}q>{n}.
\end{pdeq}
\end{align*}
Then 
\begin{align*}
\norm{\uvel}_{\LRper{r_0}\np{\R;\LR{p_0}(\Omega)}}
+\norm{\grad\uvel}_{\LRper{r_1}\np{\R;\LR{p_1}(\Omega)}}  \leq
\Cc{C}\norm{\uvel}_{\WSRper{1}{p}(\R; \LR{p}(\Omega)) \cap \LRper{p}\left(\R; \WSR{2}{p}(\Omega)\right)},
\end{align*}
with $\Cclast{C}=\Cclast{C}(\per,\Omega,r_0,p_0,r_1,p_1)$.
\end{lem}
\begin{proof}
See \cite[Theorem 4.1]{GaldiKyed}.
\end{proof}

Three types of trace operators are employed in the following:
\begin{align*}
&\TDN:\CRiper(\R\times\overline{\Omega})\ra\CRper(\R\times\partial\Omega),&&\TDN(\uvel):=\uvel_{|{\R\times\partial\Omega}},\\
&\TD:\CRiper(\R\times\overline{\Omega})\ra\CRper(\R\times\partial\Omega)^2,&& 
\TD(\uvel):=\left(\uvel_{|{\R\times\partial\Omega}}, \, \Delta\uvel_{|{\R\times\partial\Omega}}\right),\\
&\TN:\CRiper(\R\times\overline{\Omega})\ra\CRper(\R\times\partial\Omega)^2,&& 
\TN(\uvel):=\left(\partial_\nu\uvel_{|{\R\times\partial\Omega}}, \, \partial_\nu{\Delta\uvel}_{|{\R\times\partial\Omega}}\right).
\end{align*}
In order to characterize appropriate trace spaces, we introduce 
\begin{align*}
\begin{aligned}
&\TRD{1}(\R\times\partial\Omega) := 
\WSRper{3-\frac{1}{2p}}{p}\bp{\R; \LR{p}(\partial\Omega)}\cap\WSRper{1}{p}\bp{\R; \WSR{4-\frac{1}{p}}{p}(\partial\Omega)},\\
&\TRD{2}(\R\times\partial\Omega) := \WSRper{2-\frac{1}{2p}}{p}\bp{\R; \LR{p}(\partial\Omega)}\cap
\WSRper{1}{p}\bp{\R; \WSR{2-\frac{1}{p}}{p}(\partial\Omega)},\\
&\TRN{1}(\R\times\partial\Omega) := \WSRper{\frac{5}{2} - \frac{1}{2p}}{p}\bp{\R; \LR{p}(\partial\Omega)}\cap
\WSRper{1}{p}\bp{\R; \WSR{3-\frac{1}{p}}{p}(\partial\Omega)}, \\
&\TRN{2}(\R\times\partial\Omega) := \WSRper{\frac{3}{2} - \frac{1}{2p}}{p}\bp{\R; \LR{p}(\partial\Omega)}\cap
\WSRper{1}{p}\bp{\R; \WSR{1-\frac{1}{p}}{p}(\partial\Omega)}.
\end{aligned}
\end{align*}
These spaces can be identified as trace spaces in the following sense:

\begin{lem}\label{TraceSpaceLemma}
Let $\Omega\subset\R^3$ be a bounded domain with a $\CR{4}$-smooth boundary. The 
trace operators extend to bounded operators:
\begin{align}
&\TDN: \PS{p}(\R\times\Omega)\to \TRD{1}(\R\times\partial\Omega),\label{TraceSpaceLemma_TDNa}\\
&\TDN: \WSRper{2}{p}\bp{\R; \LR{p}(\Omega)} \cap \WSRper{1}{p}\bp{\R; \WSR{2}{p}(\Omega)} \to \TRD{2}(\R\times\partial\Omega),\label{TraceSpaceLemma_TDNb}\\
&\begin{aligned}
&\TDN: \WSRper{2}{p}\bp{\R; \LR{p}(\Omega)} \cap \LRper{p}\bp{\R; \WSR{4}{p}(\Omega)} \\
&\qquad\qquad\qquad\qquad\qquad\qquad
\to \WSRper{2-\frac{1}{2p}}{p}\bp{\R; \LR{p}(\partial\Omega)}\cap\LRper{p}\bp{\R;\WSR{4-\frac{1}{p}}{p}(\partial\Omega)}, 
\end{aligned}\label{TraceSpaceLemma_TDNc}\\
&\TD: \PS{p}(\R\times\Omega) \to \TRD{1}(\R\times\partial\Omega) \times \TRD{2}(\R\times\partial\Omega),\label{TraceSpaceLemma_TD}\\
&\TN: \PS{p}(\R\times\Omega) \to \TRN{1}(\R\times\partial\Omega) \times \TRN{2}(\R\times\partial\Omega).\label{TraceSpaceLemma_TN}
\end{align}
Moreover, the operators above possess a continuous right-inverse. By $\TDNrinv$ we denote a continuous right-inverse to $\TDN$.
\end{lem}

\begin{proof}
It suffices to verify the assertions in the half space case $\Omega:=\R^3_+$. The general case of a bounded domain $\Omega$ with a $\CR{4}$-smooth boundary then follows via localization. 
Observe that
\begin{align*}
\PS p(\R\times\R^3_+) &=  \LR{p}\bp{\R_+; \PS p(\R\times\R^2)}\cap\WSR{4}{p}\bp{\R_+; \WSRper{1}{p}\bp{\R; \LR{p}(\R^2)}} \\
&\quad\cap\WSR{3}{p}\bp{\R_+; \WSRper{1}{p}\bp{\R; \WSR{1}{p}(\R^2)}}\cap\WSR{2}{p}\bp{\R_+; \WSRper{1}{p}\bp{\R; \WSR{2}{p}(\R^2)}} \\
&\quad\cap\WSR{1}{p}\bp{\R_+; \WSRper{1}{p}\bp{\R; \WSR{3}{p}(\R^2)}} \\
&=  \LR{p}\bp{\R_+; \PS p(\R\times\R^2)}\cap\WSR{4}{p}\bp{\R_+; \WSRper{1}{p}\bp{\R; \LR{p}(\R^2)}},
\end{align*}
where the last equality is due to the embeddings \eqref{EmbeddingLemEmb5}.
It follows from \cite[Theorem 1.8.3]{TriebelInterpolation} that $\TDN$ extends to continuous operator
\begin{align*}
\TDN: \PS p(\R\times\R^3_+) \to \left(\WSRper{1}{p}(\R; \LR{p}(\R^2)), \, \PS{p}(\R\times\R^2)\right)_{1-1/4p, \, p}.
\end{align*}
One may verify that $\PS p(\R\times\R^2)$ and $\WSRper{1}{p}\bp{\R; \LR{p}(\R^2)}$ 
form a quasilinearizable interpolation couple; see  
\cite[Definition 1.8.4]{TriebelInterpolation}. Indeed,
an admissible operator  
in the sense of 
\cite[Definition 1.8.4]{TriebelInterpolation} 
is given by $V_1(\mu):=\mu^{-1}\bp{\mu^{-1}-\partial_t^2+\Delta^2}^{-1}$,
where invertibility of $\bp{\mu^{-1}-\partial_t^2+\Delta^2}:\PS p(\R\times\R^2)\ra\WSRper{1}{p}\bp{\R; \LR{p}(\R^2)}$ can be established by an analysis of the multiplier $(\eta,\xi)\ra\bp{\mu^{-1}+\eta^2+\snorm{\xi}^4}^{-1}$ and an application of de Leeuw's Transference Principle 
as in the proof of Lemma \ref{EmbeddingLemma}.
Consequently, one obtains even better properties of the trace operator, namely that it possesses a continuous right inverse; see \cite[Theorem 1.8.5]{TriebelInterpolation}. Moreover, we can utilize the property stated in \cite[Theorem 1.12.1]{TriebelInterpolation} concerning interpolation of intersections 
of spaces that form quasilinearizable interpolation couples to conclude 
\begin{align*}
\left(\WSRper{1}{p}(\R; \LR{p}(\R^2)), \, \PS{p}(\R\times\R^2)\right)_{1-1/4p, \, p} = \TRD 1(\R\times\R^2),
\end{align*}
which verifies \eqref{TraceSpaceLemma_TDNa}. The assertions \eqref{TraceSpaceLemma_TDNb}--\eqref{TraceSpaceLemma_TN} follow in a similar way.
\end{proof}

\begin{lem}\label{EmbeddingLemma2}
Let $\Omega\subset\R^3$ be a bounded domain with a $C^{4}$-smooth boundary. The embedding 
\begin{align}
&\WSRper{2-\frac{1}{2p}}{p}\bp{\R; \LR{p}(\partial\Omega)}\cap\LRper{p}\bp{\R;\WSR{4-\frac{1}{p}}{p}(\partial\Omega)}
\embeds \TRD{2}(\R\times\partial\Omega)\label{EmbeddingLemEmb3}
\end{align}
is continuous.
\end{lem}

\begin{proof}
Denote the embedding \eqref{EmbeddingLemEmb2} by $\iota$. Recalling \eqref{TraceSpaceLemma_TDNb} and \eqref{TraceSpaceLemma_TDNc},
we find that $\TDN\circ\iota\circ\TDNrinv$ yields the embedding \eqref{EmbeddingLemEmb3}.
\end{proof}

For functions $f$ defined on time-space domains, we let
\begin{equation*}
\PR f (t,x) := \frac{1}{\per}\int_0^\per f(s,x)\,\ds,\quad \oPR f(t,x):= f(t,x)-\proj f(t,x)
\end{equation*}
whenever the integral is well defined. Since $\PR f$ is independent on time $t$, we shall implicitly treat $\proj f$ as a function 
in the spatial variable $x$ only. 
Observe that $\proj$ and $\projcompl$ are complementary projections on the space $\CRiper\bp{\R; E(\Omega)}$.
We shall employ the projections 
to decompose the Lebesgue and Sobolev spaces introduced above. Since $\PR f$ is time independent, 
we refer to $\PR f$ as the \emph{steady-state} part of $f$. We refer to $\oPR f$ as the \emph{purely oscillatory} part of $f$.
By continuity, $\PR$ and $\oPR$ extend to bounded operators on 
$\LRper{p}\bp{\R; E(\Omega)}$ and $\WSRper{k}{p}\bp{\R; E(\Omega)}$.

\section{Linear Problem}\label{lin}

The linear equations \eqref{LinBCD} and \eqref{LinBCN} can be decomposed into a time-periodic Kuznetsov equation
coupled with a time-periodic heat equation. Based on this observation, we obtain the following linear theory:

\begin{thm}\label{homeo}
Assume that $\Omega\subset\R^3$ is a bounded domain with a $C^{4}$-smooth boundary. 
Let $p\in (1, \infty)$. Then
\begin{align*}
&\operatorname{A_D}\colon\oPR\PS{p}(\R\times\Omega)\to \\
&\qquad\qquad\qquad \oPR\LRper{p}\bp{\R; \LR{p}(\Omega)}\times\oPR \TRD{1}(\R\times\partial\Omega)\times\oPR \TRD{2}(\R\times\partial\Omega), \\
&\operatorname{A_D}(\uvel) := \left(\ADdiffopru, \operatorname{Tr}_D\uvel\right)
\end{align*}
and 
\begin{align*}
&\operatorname{A_N}\colon \oPR\PS{p}(\R\times\Omega) \to\\
&\qquad\qquad\qquad \oPR\LRper{p}\bp{\R; \LR{p}(\Omega)}\times\oPR \TRN{1}(\R\times\partial\Omega)\times\oPR \TRN{2}(\R\times\partial\Omega), \\
&\operatorname{A_N}(\uvel) := \left(\ADdiffopru, \operatorname{Tr}_N\uvel\right), 
\end{align*}
are homeomorphisms.
\end{thm}

\begin{proof}
On the strength of the embedding \eqref{EmbeddingLemEmb1}, we observe that $\Delta$ is a bounded operator
\begin{align*}
\Delta: \oPR\PS{p}(\R\times\Omega)
\ra \oPR\WSRper{2}{p}\bp{\R; \LR{p}(\Omega)}\cap\oPR\WSRper{1}{p}\bp{\R;\WSR{2}{p}(\Omega)}.
\end{align*}
Together with the continuity of the trace operators established in  
Lemma \ref{TraceSpaceLemma}, this implies that the operators $\operatorname{A_D}$ and $\operatorname{A_N}$ are well-defined as bounded 
operators in the given setting. We start by showing that the operators are surjective. We concentrate on $\operatorname{A_D}$, as the operator  
$\operatorname{A_N}$ can be treated in a completely similar manner. To this end, let 
\begin{align*}
(f,g,h)\in 
\oPR\LRper{p}\bp{\R; \LR{p}(\Omega)}\times
\oPR \TRD{1}(\R\times\partial\Omega)\times\oPR \TRD{2}(\R\times\partial\Omega).
\end{align*}
Consider the coupled equations
\begin{align}\label{damp_waveD}
\begin{pdeq}
\partial_{t}^2 v - c^2\Delta v - b\partial_t\Delta v &= f && \tin\R\times\Omega, \\
v &= ah - \partial_t g && \ton\R\times\partial\Omega,
\end{pdeq}
\end{align}
and
\begin{align}\label{HeatD}
\begin{pdeq}
a\Delta\uvel - \partial_t\uvel &= v && \tin\R\times\Omega, \\
\uvel &= g && \ton\R\times\partial\Omega.
\end{pdeq}
\end{align}
We recognize \eqref{damp_waveD} as the time-periodic Kuznetsov equation, which was studied in \cite{CelikKyed_nweqwd}, and
\eqref{HeatD} as the time-periodic heat equation, which was studied in \cite{KyedSauer_Heat}.
Recalling the embedding \eqref{EmbeddingLemEmb3}, we see that $\partial_t$ is bounded as an operator
\begin{equation}\label{homeoThmPartialtMappingProp1}
\partial_t:\oPR \TRD{1}(\R\times\partial\Omega) \ra \oPR \TRD{2}(\R\times\partial\Omega),
\end{equation}
whence $ah - \partial_t g\in \oPR \TRD{2}(\R\times\partial\Omega)$. Consequently, we obtain directly from
\cite[Theorem 3.1]{CelikKyed_nweqwd} existence of a unique solution 
\begin{equation*}
v\in\oPR\WSRper{2}{p}\bp{\R; \LR{p}(\Omega)}\cap\oPR\WSRper{1}{p}\bp{\R; \WSR{2}{p}(\Omega)}
\end{equation*}
to \eqref{damp_waveD}. We now turn to \eqref{HeatD}. 
From \cite[Theorem 2.1]{KyedSauer_Heat} and a standard regularity and lifting argument, based on the mapping property 
of the trace operator $\TDN$ established in Lemma \ref{TraceSpaceLemma}, we obtain a unique solution $\uvel\in\oPR\PS{p}(\R\times\Omega)$ to \eqref{HeatD}. Recalling the embedding 
\eqref{EmbeddingLemEmb2}, we see that $\partial_t$ is bounded as an operator
\begin{align}\label{homeoThmPartialtMappingProp2}
\partial_t: \oPR\PS{p}(\R\times\Omega) \ra \oPR\WSRper{2}{p}\bp{\R; \LR{p}(\Omega)}\cap\oPR\WSRper{1}{p}\bp{\R;\WSR{2}{p}(\Omega)}.
\end{align}
Since the operators $\partial_t$ and $\TDN$ commute on spaces of smooth functions, it follows 
from \eqref{homeoThmPartialtMappingProp1}, \eqref{homeoThmPartialtMappingProp2} and the mapping property of the trace operator $\TDN$ asserted in Lemma \ref{TraceSpaceLemma} that they commute in the setting
\begin{equation*}
\partial_t\circ\TDN = \TDN\circ\partial_t :
\oPR\PS{p}(\R\times\Omega) \ra\oPR \TRD{2}(\R\times\partial\Omega).
\end{equation*} 
We thus deduce from \eqref{HeatD} and the boundary condition in \eqref{damp_waveD} that
\begin{align*}
\TD\uvel = \bp{g,\TDN\np{\Delta \uvel}} 
= \Bp{g,\TDN\bp{ \frac{1}{a}\bb{\partial_t\uvel+v}}} 
= \bp{g,\frac{1}{a}\partial_t g+\frac{1}{a}v} = \bp{g,h}. 
\end{align*}
It follows that $\operatorname{A_D}(\uvel)=(f,g,h)$, and we conclude that $\operatorname{A_D}$ is surjective.
To show that $\operatorname{A_D}$ is injective, consider $\uvel\in\oPR\PS{p}(\R\times\Omega)$ with $\operatorname{A_D}(\uvel)=(0,0,0)$.
Unique solvability of the time-periodic Kuznetsov equation \cite[Theorem 3.1]{CelikKyed_nweqwd} implies 
$a\Delta\uvel - \partial_t\uvel=0$. In turn,  
unique solvability of the time-periodic heat equation \cite[Theorem 2.1]{KyedSauer_Heat} implies $\uvel=0$.
Consequently, $\operatorname{A_D}$ is injective. By the open mapping theorem, $\operatorname{A_D}$ is a homeomorphism.
\end{proof}

Next, we recall some standard results for the bi-Laplacian Dirichlet problem
\begin{align}\label{BCDS}
\begin{pdeq}
-ac^2\Delta^2\uvel &= f && \tin\Omega, \\
\bp{\uvel, \Delta\uvel} &= \bp{g, h} && \ton\partial\Omega,
\end{pdeq}
\end{align}
and the corresponding Neumann problem
\begin{align}\label{BCNS}
\begin{pdeq}
-ac^2\Delta^2\uvel &= f && \tin\Omega, \\
\bp{\partial_\nu\uvel, \partial_\nu{\Delta\uvel}} &= \bp{g, h} && \ton\partial\Omega.
\end{pdeq}
\end{align}

\begin{lem}\label{DSBD}
Let $\Omega$ and $p$ be as in Theorem \ref{homeo}. For any $f\in\LR{p}\np{\Omega}$, $g\in\WSR{4-\frac{1}{p}}{p}\np{\partial\Omega}$ and $h\in\WSR{2-\frac{1}{p}}{p}\np{\partial\Omega}$ there exists a unique solution $\uvel\in\WSR{4}{p}\np{\Omega}$ to \eqref{BCDS} satisfying
\begin{align}\label{reg_Dirichlet_BDS}
\norm{\uvel}_{4, p}\leq \Cc[EstSsBdD]{C}\left(\norm{f}_p + \norm{g}_{4-\frac{1}{p}, p} + \norm{h}_{2-\frac{1}{p}, p}\right),
\end{align}
with $\const{EstSsBdD} = \const{EstSsBdD}(p, \Omega) > 0$.
\end{lem}

\begin{proof}
Considering the coupled equations 
\begin{align*}
\begin{pdeq}
\Delta v &= f && \tin\Omega, \\
v &= -ac^2h && \ton\partial\Omega,
\end{pdeq}
\qquad
\begin{pdeq}
-ac^2\Delta\uvel &= v && \tin\Omega, \\
\uvel &= g && \ton\partial\Omega,
\end{pdeq}
\end{align*}
one easily obtains the assertion of the lemma from standard theory for the Laplace equation with Dirichlet boundary conditions.
\end{proof}

\begin{lem}\label{NSBD}
Let $\Omega$ and $p$ be as in Theorem \ref{homeo}. For any $f\in\LR{p}(\Omega)$, $g\in\WSR{3-\frac{1}{p}}{p}(\partial\Omega)$ and $h\in\WSR{1-\frac{1}{p}}{p}(\partial\Omega)$ satisfying
\begin{align*}
\int_\Omega f \,\dx + ac^2\int_{\partial\Omega} h \,\dS = 0
\end{align*}
there exists a unique solution $\uvel\in\WSR{4}{p}(\Omega)$ to \eqref{BCNS} satisfying
\begin{align*}
\norm{\uvel}_{4, p}\leq \Cc[EstSsBdN]{C}\left(\norm{f}_p + \norm{g}_{3-\frac{1}{p}, p} + \norm{h}_{1-\frac{1}{p}, p}\right),
\end{align*}
with $\const{EstSsBdN} = \const{EstSsBdN}(p, \Omega) > 0$.
\end{lem}

\begin{proof}
As in the proof of Lemma \ref{DSBD}, the assertion follows by decomposing
\eqref{BCNS} into two Laplace equations with Neumann boundary conditions.
\end{proof}

\section{Proof of Main Theorems}\label{Main}

We shall now prove Theorem \ref{max_regBCD} and \ref{max_regBCN}. For this purpose, we employ a fixed-point argument based on the estimates established for the linearized systems \eqref{LinBCD} and \eqref{LinBCN} in the previous section. 

The nonlinear terms in \eqref{BCD} and \eqref{BCN} can be estimated as follows:

\begin{lem}\label{nonlin}
Let $\Omega$ and $p$ be as in Theorem \ref{max_regBCD}. Then 
\begin{align*}
\norm{\partial_t v \, \partial_t^3\uvel}_p + \norm{\partial_t^2 v\, \partial_t^2\uvel}_p + \norm{\partial_t\grad v\cdot\partial_t\grad\uvel}_p + \norm{\grad v\cdot\partial_t^2\grad\uvel}_p \leq \Cc[NLE]{C}\norm{v}_{\PS{p}}\norm{\uvel}_{\PS{p}}
\end{align*}
holds for any $\uvel, v\in\PS{p}(\R\times\Omega)$. 
\end{lem}

\begin{proof}
For any function $\wvel\in\PS{p}(\R\times\Omega)$ we clearly have
\begin{align*}
\norm{\partial_t\wvel, \partial_t^2\wvel, \grad\wvel}_{\WSRper{1}{p}(\R; \LR{p}(\Omega)) \cap \LRper{p}(\R; \WSR{2}{p}(\Omega))} \leq \norm{\wvel}_{\PS{p}(\R\times\Omega)}.
\end{align*}
Thus, for $p\in (\frac{5}{2}, \infty)$ Lemma \ref{SobEmbeddingThm} yields 
\begin{align*}
\alpha &= \frac{4}{5}:\qquad \norm{\partial_t v}_{\infty} + \norm{\partial_t^2 v}_{\infty} \leq \Cc[alpha0]{c}\norm{v}_{\PS{p}}, \\
\alpha &= 0:\qquad \norm{\grad v}_{\LRper{p}\np{\R; \LR{\infty}(\Omega)}} \leq \Cc[alpha1]{c}\norm{v}_{\PS{p}},
\end{align*}
and for $p\in(\frac{5}{2},\frac{15}{4})$
\begin{align}\label{nonlinEmb2}
\begin{aligned}
\beta &= 1:\qquad \norm{\partial_t^2\grad\uvel}_{\LRper{\infty}\np{\R; \LR{p}(\Omega)}} \leq \Cc[beta0]{c}\norm{\uvel}_{\PS{p}}, \\
\beta &= \frac{1}{5}:\qquad \norm{\partial_t\grad v}_{\LRper{\frac{10p}{10-p}}\np{\R; \LR{\frac{15p}{15-4p}}(\Omega)}} \leq \Cc[beta1]{c}\norm{v}_{\PS{p}}, \\
\beta &= \frac{3}{5}: \qquad\norm{\partial_t\grad\uvel}_{\LRper{10}\np{\R; \LR{\frac{15}{4}}(\Omega)}} \leq \Cc[beta2]{c}\norm{\uvel}_{\PS{p}}.
\end{aligned}
\end{align}
Utilizing H\"older's inequality, we can therefore deduce
\begin{align*}
&\norm{\grad v\cdot\partial_t^2\grad\uvel}_p + \norm{\partial_t v\,\partial_t^3\uvel}_p + \norm{\partial_t^2 v\,\partial_t^2\uvel}_p + \norm{\partial_t\grad v\cdot\partial_t\grad\uvel}_p \\
&\quad\leq \norm{\grad v}_{\LRper{p}\np{\R; \LR{\infty}(\Omega)}}\norm{\partial_t^2\grad\uvel}_{\LRper{\infty}\np{\R; \LR{p}(\Omega)}} + \norm{\partial_t v}_{\infty}\norm{\partial_t^3\uvel}_{p} \\ 
&\quad\quad + \norm{\partial_t^2 v}_{\infty}\norm{\partial_t^2\uvel}_{p} + \norm{\partial_t\grad v}_{\LRper{\frac{10p}{10-p}}\np{\R; \LR{\frac{15p}{15-4p}}(\Omega)}}\norm{\partial_t\grad\uvel}_{\LRper{10}\np{\R; \LR{\frac{15}{4}}(\Omega)}} \\
&\quad\leq \const{NLE}\norm{v}_{\PS{p}}\norm{\uvel}_{\PS{p}},
\end{align*}
with $\const{NLE} = 2\const{alpha0} + \const{alpha1}\const{beta0} + \const{beta1}\const{beta2}$. 
The estimates remain true for $p\in[\frac{15}{4},\infty)$, in which case even stronger embeddings in \eqref{nonlinEmb2} follow from Lemma \ref{SobEmbeddingThm}. 
\end{proof}

\begin{proof}[Proof of Theorem \ref{max_regBCD}]
Consider functions $f\in\LRper{p}(\R; \LR{p}\left(\Omega\right))$, $g\in\TRD{1}(\R\times\partial\Omega)$ and $h\in\TRD{2}(\R\times\partial\Omega)$ with 
$\norm{f}_p + \norm{g}_{\TRD{1}} + \norm{h}_{\TRD{2}} \leq \varepsilon$, where $\varepsilon$ is to be specified later. 
Put $f_s := \PR f$, $g_s := \PR g$, $h_s := \PR h$ and $f_p := \oPR f$, $g_p := \oPR g$, $h_p := \oPR h$.
We shall establish existence of a solution $u$ to \eqref{BCD} on the form $u = \us + \up$, where $\us\in\WSR{4}{p}(\Omega)$ is a solution to the elliptic problem \eqref{BCDS} with $f_s$ and $(g_s,h_s)$ as right-hand side, and $\up\in\oPR\PS{p}(\R\times\Omega)$ is a solution to the purely oscillatory problem
\begin{align}\label{periodic_BCD}
\begin{pdeq}
(a\Delta - \partial_t)(\partial_{t}^2\up-c^2\Delta\up - b\partial_t\Delta\up) \qquad\qquad\quad \\ 
- 2s\grad\us\cdot\grad\partial_t^2\up - \partial_t^2\bp{k\np{\partial_t\up}^2 + s\snorm{\grad\up}^2} &= f_p && \tin\R\times\Omega, \\
\bp{\up, \Delta\up} &= \bp{g_p, h_p} && \ton\R\times\partial\Omega.
\end{pdeq}
\end{align}
Lemma \ref{DSBD} yields a solution $\us\in\WSR{4}{p}(\Omega)$ to \eqref{BCDS} satisfying \eqref{reg_Dirichlet_BDS}. Sobolev's embedding theorem implies 
\begin{align}\label{SobolevSS}
\norm{\grad\us}_{\infty}\leq \Cc[Steady0]{c}\norm{\grad\us}_{3, p} \leq \const{Steady0}\norm{\us}_{4, p} \leq \Cc{c} \epsilon.
\end{align}
The solution to \eqref{periodic_BCD} shall be obtained as a fixed point of the mapping
\begin{align*}
&\mathcal{N}\colon\oPR\PS{p}(\R\times\Omega)\to\oPR\PS{p}(\R\times\Omega), \\
&\mathcal{N}\left(\up\right) := \operatorname{A_D}^{-1}\left(\partial_t^2\bp{k\np{\partial_t\up}^2 + s\snorm{\grad\up}^2} + 2s\grad\us\cdot\partial_t^2\grad\up + f_p, g_p, h_p\right)
\end{align*}
with $\operatorname{A_D}$ as in Theorem \ref{homeo}. For this purpose, let $\rho > 0$ and consider some $\up\in \oPR\PS{p}(\R\times\Omega)\cap B_\rho$. Since $\operatorname{A_D}$ is a homeomorphism, we conclude from Lemma \ref{nonlin} and \eqref{SobolevSS} the estimate
\begin{align*}
\norm{\mathcal{N}\left(\up\right)}_{\PS{p}} &\leq \norm{\operatorname{A_D}^{-1}}\Big(\norm{\partial_t \up\partial_t^3\up}_p + \norm{\partial_t^2 \up\partial_t^2\up}_p + \norm{\partial_t\grad \up\cdot\partial_t\grad\up}_p + \norm{f_p}_p \\
&\quad + \norm{\grad \up\cdot\partial_t^2\grad\up}_p + \norm{\grad\us\cdot\grad\partial_t^2\up}_p + \norm{g_p}_{\TRD{1}} + \norm{h_p}_{\TRD{2}}\Big) \\
&\leq \Cc{c}(\norm{\up}_{\PS{p}}^2 + \norm{\grad\us}_{\LR{\infty}(\Omega)}\norm{\grad\partial_t^2\up}_{\LRper{p}\np{\R; \LR{p}(\Omega)}} + \varepsilon) \\ 
&\leq \Cc{c}(\rho^2 + \varepsilon\rho + \varepsilon).
\end{align*}
Choosing $\rho:=\sqrt{\varepsilon}$ and $\epsilon$ sufficiently small, we have $\Cclast{c}\left(\rho^2 + \varepsilon\rho + \varepsilon\right) \leq \rho,$ \textit{i.e.}, $\mathcal{N}$ becomes a self-mapping on $B_\rho$. Furthermore,
\begin{align*}
\norm{\mathcal{N}&(\up) - \mathcal{N}(v_p)}_{\PS{p}} \leq \Cc{c}\norm{\operatorname{A_D}^{-1}}\Big(\norm{\partial_t\up\,\partial_t^3\up - 
\partial_t v_p\,\partial_t^3 v_p}_p \\ 
&\quad + \norm{\snorm{\grad\partial_t\up}^2 - \snorm{\grad\partial_t v_p}^2}_p + \norm{\grad\up\cdot\partial_t^2\grad\up - \grad v_p\cdot\partial_t^2\grad v_p}_p \\
&\quad + \norm{(\partial_t^2\up)^2 - (\partial_t^2 v_p)^2}_p + \norm{\grad\us\cdot\grad\partial_t^2\up - \grad\us\cdot\grad\partial_t^2 v_p}_p \Big) \\
&\leq \Cc{c}\Big(\norm{\partial_t\up\partial_t^3\left(\up - v_p\right)}_p + \norm{\partial_t^3 v_p\,\partial_t(\up - v_p)}_p + \norm{\grad\partial_t\up\cdot\grad\partial_t\left(\up - v_p\right)}_p \\
&\quad + \norm{\grad\partial_t v_p\cdot\grad\partial_t\left(\up - v_p\right)}_p + \norm{\grad\up\cdot\grad\partial_t^2\left(\up - v_p\right)}_p + \norm{\grad\partial_t^2 v_p\cdot\grad\left(\up - v_p\right)}_p \\ 
&\quad + \norm{\partial_t^2\up\,\partial_t^2\left(\up - v_p\right)}_p + \norm{\partial_t^2 v_p\,\partial_t^2\left(\up - v_p\right)}_p + \norm{\grad\us\cdot\grad\partial_t^2\left(\up - v_p\right)}_p\Big) \\
&\leq \Cc[Steady0]{c}\left(8\rho\norm{\up - v_p}_{\PS{p}} + \varepsilon\norm{\up - v_p}_{\PS{p}}\right) = \const{Steady0}(8\rho + \rho^2) \norm{\up - v_p}_{\PS{p}}.
\end{align*}
Therefore, if $\epsilon$ is sufficiently small $\mathcal{N}$ also becomes a contracting self-mapping. By the contraction mapping principle, existence of a fixed point for $\mathcal{N}$ follows. This concludes the proof.
\end{proof}

\begin{proof}[Proof of Theorem \ref{max_regBCN}]
Analogous to the proof of Theorem \ref{max_regBCD}.
\end{proof}

\bibliographystyle{plainurl}

\begin{thebibliography}{10}

\bibitem{Bla63}
D.T. Blackstock.
\newblock Approximate equations governing finite-amplitude sound in
  thermoviscous fluids.
\newblock GD/E report GD-1463-52, General Dynamics Coporation, 1963.

\bibitem{Bru15}
Rainer Brunnhuber.
\newblock Well-posedness and exponential decay of solutions for the
  {B}lackstock-{C}righton-{K}uznetsov equation.
\newblock {\em J. Math. Anal. Appl.}, 433(2):1037--1054, 2016.
\newblock \href {http://dx.doi.org/10.1016/j.jmaa.2015.07.046}
  {\path{doi:10.1016/j.jmaa.2015.07.046}}.

\bibitem{BK14}
Rainer Brunnhuber and Barbara Kaltenbacher.
\newblock Well-posedness and asymptotic behavior of solutions for the
  {B}lackstock-{C}righton-{W}estervelt equation.
\newblock {\em Discrete Contin. Dyn. Syst.}, 34(11):4515--4535, 2014.
\newblock \href {http://dx.doi.org/10.3934/dcds.2014.34.4515}
  {\path{doi:10.3934/dcds.2014.34.4515}}.

\bibitem{BM16}
Rainer Brunnhuber and Stefan Meyer.
\newblock Optimal regularity and exponential stability for the                                                                       
  {B}lackstock-{C}righton equation in {$L_p$}-spaces with {D}irichlet and                                                            
  {N}eumann boundary conditions.                                                                                                     
\newblock {\em J. Evol. Equ.}, 16(4):945--981, 2016.                                                                                 
\newblock \href {http://dx.doi.org/doi.org/10.1007/s00028-016-0326-6}                                                                
  {\path{doi:doi.org/10.1007/s00028-016-0326-6}}.                                                                                    
                                                                                                                                     
\bibitem{CelikKyed_nweqwd}                                                                                                           
Aday Celik and Mads Kyed.                                                                                                            
\newblock {Nonlinear Wave Equation with Damping: Periodic Forcing and                                                                
  Non-Resonant Solutions to the Kuznetsov Equation}.                                                                                 
\newblock {\em {Z. Angew. Math. Mech}}, 1-19, 2017.                                                                                  
\newblock \href {http://dx.doi.org/10.1002/zamm.201600280}                                                                           
  {\path{doi:10.1002/zamm.201600280}}.                                                                                               
                                                                                                                                     
\bibitem{EdG77}                                                                                                                      
R.~E. Edwards and G.~I. Gaudry.                                                                                                      
\newblock {\em Littlewood-{P}aley and multiplier theory}.                                                                            
\newblock Springer-Verlag, Berlin-New York, 1977.                                                                                    
\newblock Ergebnisse der Mathematik und ihrer Grenzgebiete, Band 90.                                                                 
                                                                                                                                     
\bibitem{GaldiKyed}                                                                                                                  
Giovanni~P. Galdi and Mads Kyed.                                                                                                     
\newblock {Time-period flow of a viscous liquid past a body}.                                                                        
\newblock {\em {To appear in London Mathematical Society Lecture Note Series}},                                                      
  2016.                                                                                                                              
\newblock URL: \url{https://arxiv.org/abs/1609.09829}.                                                                               
                                                                                                                                     
\bibitem{KyedSauer_Heat}                                                                                                             
Mads {Kyed} and Jonas {Sauer}.                                                                                                       
\newblock {A method for obtaining time-periodic $L^{p}$ estimates.}                                                                  
\newblock {\em {J. Differ. Equations}}, 262(1):633--652, 2017.                                                                       
\newblock \href {http://dx.doi.org/10.1016/j.jde.2016.09.037}                                                                        
  {\path{doi:10.1016/j.jde.2016.09.037}}.                                                                                            
                                                                                                                                     
\bibitem{Stein70}                                                                                                                    
Elias~M. Stein.                                                                                                                      
\newblock {\em Singular integrals and differentiability properties of                                                                
  functions}.                                                                                                                        
\newblock Princeton Mathematical Series, No. 30. Princeton University Press,                                                         
  Princeton, N.J., 1970.                                                                                                             
                                                                                                                                     
\bibitem{Tani}                                                                                                                       
Atusi Tani.                                                                                                                          
\newblock Mathematical analysis in nonlinear acoustics.
\newblock {\em AIP Conference Proceedings}, 1907(1):020003, 2017.
\newblock \href {http://dx.doi.org/10.1063/1.5012614}
  {\path{doi:10.1063/1.5012614}}.

\bibitem{TriebelInterpolation}
H.~Triebel.
\newblock {\em Interpolation theory, function spaces, differential operators}.
\newblock VEB Deutscher Verlag der Wissenschaften, Berlin, 1978.

\end{thebibliography}

\end{document}